\documentclass[secthm,seceqn,amsthm,ussrhead,12pt]{amsart}
\usepackage{amsmath,latexsym}
\usepackage[english]{babel}
\usepackage[psamsfonts]{amssymb}
\usepackage{times}
\usepackage{cite}

\usepackage[mathcal]{euscript}
\numberwithin{equation}{section} \textwidth=17.5cm
\newtheorem{theorem}{Theorem}[section]
\newtheorem{lemma}[theorem]{Lemma}
\newtheorem{example}[theorem]{Example}
\newtheorem{proposition}[theorem]{Proposition}
\newtheorem{corollary}[theorem]{Corollary}
\newtheorem{definition}[theorem]{Definition}
\newtheorem{remark}[theorem]{Remark}

\numberwithin{equation}{section}

\begin{document}

\title{2-Local   derivations
on matrix algebras over commutative regular algebras}

 \author[Shavkat Ayupov]{Shavkat Ayupov}
 \author[Karimbergen Kudaybergenov]{Karimbergen Kudaybergenov}
 \author[Amir Alauadinov]{Amir Alauadinov}
\address[Shavkat Ayupov]{Institute of
 Mathematics, National University of
Uzbekistan,
 100125  Tashkent,   Uzbekistan
 and
 the Abdus Salam International Centre
 for Theoretical Physics (ICTP),
  Trieste, Italy,
\textit{sh$_{-}$ayupov@mail.ru}}

\address[Karimbergen Kudaybergenov]{Department of Mathematics
 Karakalpak state university
Ch. Abdirov 1
  230113, Nukus,    Uzbekistan,
  \textit{karim2006@mail.ru}}

  \address[Amir Alauadinov]{Institute of
 Mathematics, National University of
 Uzbekistan,
 100125  Tashkent,   Uzbekistan,
\textit{amir\_t85@mail.ru}}

\begin{abstract}
The paper is devoted to $2$-local derivations on matrix algebras
over commutative regular algebras. We give necessary and
sufficient conditions on a commutative regular algebra
 to admit $2$-local
derivations which are not derivations. We prove that every
$2$-local derivation on a matrix algebra over a commutative
regular algebra is a derivation. We apply these results to
$2$-local derivations on algebras of measurable and locally
measurable operators affiliated with type I von Neumann algebras.
\end{abstract}


\maketitle
\section{Introduction}

Given an algebra $\mathcal{A},$ a linear operator
$D:\mathcal{A}\rightarrow \mathcal{A}$ is called a
\textit{derivation}, if $D(xy)=D(x)y+xD(y)$ for all $x, y\in
\mathcal{A}$ (the Leibniz rule). Each element $a\in \mathcal{A}$
implements a derivation $D_a$ on $\mathcal{A}$ defined as
$D_a(x)=[a, x]=ax-xa,$ $x\in \mathcal{A}.$ Such derivations $D_a$
are said to be \textit{inner derivations}. If the element $a,$
implementing the derivation $D_a,$ belongs to a larger algebra
$\mathcal{B}$ containing $\mathcal{A},$ then $D_a$ is called
\textit{a spatial derivation} on $\mathcal{A}.$

If the algebra $\mathcal{A}$ is commutative, then it is clear that
all inner derivations are trivial, i.e. identically zero. One of
the main problems concerning derivations is to prove that every
derivation on a certain algebra is inner or spatial, or to show
the existence on a given algebra of non inner (resp. non spatial)
derivations, in particular the existence of non trivial
derivations in the commutative case.

In the paper \cite{Ber} A. F. Ber, V. I. Chilin and F. A. Sukochev
obtained necessary and sufficient conditions for the existence of
non trivial derivations on regular commutative algebras. In
particular they have proved that the algebra $L^{0}(0,1)$ of all
(equivalence classes of) complex measurable function on the
$(0,1)$ interval admits non trivial derivations.  It is clear that
such derivations are discontinuous and non inner. We have
conjectured in \cite{Alb2} that the existence of such "exotic"
examples of derivations is closely connected with the commutative
nature of these algebras. This was confirmed for the particular
case of type I von Neumann algebras in \cite{Alb2},
 moreover we have investigated and completely described
derivations on the algebra $LS(M)$ of all locally measurable
operators affiliated with a type I or a type III von Neumann
algebra $M$ and on its various subalgebras \cite{AKop}.

There exist various types of linear operators which are close to
derivations \cite{Kad,  Lar,  Sem1}. In particular R.~Kadison
\cite{Kad} has introduced and investigated so-called local
derivations on von Neumann algebras and some polynomial algebras.

A linear operator $\Delta$ on an algebra $\mathcal{A}$ is called a
\textit{local derivation} if given any $x\in \mathcal{A}$ there
exists a derivation $D$ (depending on $x$) such that
$\Delta(x)=D(x).$  The main problems concerning this notion are to
find conditions under which local derivations become derivations
and to present examples of algebras with local derivations that
are not derivations. In particular Kadison \cite{Kad} has proved
that each continuous local derivation from a von Neumann algebra
$M$ into a dual $M$-bimodule is a derivation. In \cite{Bre1} it
was proved that every local derivation  on the algebra $M_{n}(R)$
is  a derivation, where $M_{n}(R)$ is the algebra of $n\times n$
matrices over a unital ring $R$ containing $1/2.$

In \cite{John}, B.~E.~Johnson has extended Kadison's result and
proved that every local derivation from a $C^{\ast}$-algebra
$\mathcal{A}$ into any Banach $\mathcal{A}$-bimodule is a
derivation. He also showed that every local derivation from
 a $C^{\ast}$-algebra $\mathcal{A}$ into any Banach $\mathcal{A}$-bimodule is continuous.
In \cite{Nur} local derivations have been investigated on the
algebra $S(M)$ of all measurable operators  with respect to a von
Neumann algebra $M$.
 In particular, it was   proved
 that for finite type I  von Neumann algebras without abelian direct summands every local derivation
 on $S(M)$ is a derivation. Moreover, in the  case of abelian von Neumann algebra $M$
  necessary
and sufficient conditions have been  obtained for the algebra
$S(M)$ to admit local derivations which are not derivations.

In 1997, P. Semrl \cite{Sem1}  introduced the concepts of
$2$-local derivations and $2$-local automorphisms. A  map
$\Delta:\mathcal{A}\rightarrow\mathcal{A}$  (not linear in
general) is called a
 $2$-\emph{local derivation} if  for every $x, y\in \mathcal{A},$  there exists
 a derivation $D_{x, y}:\mathcal{A}\rightarrow\mathcal{A}$
such that $\Delta(x)=D_{x, y}(x)$  and $\Delta(y)=D_{x, y}(y).$
Local and $2$-local maps have been studied on different operator
algebras by many authors \cite{Nur, AyAr, JMAA, AKNA, AKA, Bre1,
John, Kad, Kim, Lar, Lin,  Liu,  Mol,  Sem1, Sem2, Zhang}.

In \cite{Sem1}, P. Semrl described $2$-local derivations and
automorphisms on the algebra $B(H)$ of
 all bounded linear operators on the infinite-dimensional
separable Hilbert space $H.$  A similar description for the
finite-dimensional case appeared later in \cite{Kim}. Recently in
\cite{JMAA} we have considered $2$-local derivations on the
algebra $B(H)$ of all linear bounded operators on an arbitrary (no
separability is assumed) Hilbert space $H$ and proved that every
$2$-local derivation on $B(H)$ is a derivation. This result has
been extended to arbitrary semi-finite von Neumann algebras in
\cite{AyAr}. J.~H.~Zhang and H.~X.~Li \cite{Zhang} described
$2$-local  derivations on symmetric
 digraph algebras  and constructed  a $2$-local derivation
 which is not a derivation on  the algebra
of all upper triangular complex $2\times 2$-matrices.

All algebras $\mathcal{A}$ considered in the present paper are
semi-prime, i.e. $a\mathcal{A}a=\{0\},\, a\in \mathcal{A},$
implies that $a = 0.$ If $\Delta:\mathcal{A}\rightarrow
\mathcal{A}$ is a $2$-local derivation then it is easy to see that
$\Delta$ is homogeneous and $\Delta(x^2)=\Delta(x)x+x\Delta(x) $
for all  $x\in \mathcal{A}.$ A linear map satisfying the above
identity is called a Jordan derivation. It is proved in
\cite[Theorem 1]{Bre2}  that any Jordan derivation on a semi-prime
algebra is a derivation. So in order to prove that a $2$-local
derivation $\Delta$ on a semi-prime algebra $\mathcal{A}$ is a
derivation it is sufficient to show that the map $\Delta$ is
additive.

This paper is devoted to $2$-local derivations on matrix algebras
over commutative regular algebras.

In Section 2 we give some basic results about regular commutative
algebras and their derivations.

Section 3 is devoted the problem of existence of  $2$-local
derivations which are not derivations on  a class of commutative
regular algebras, which include the algebras of measurable
functions on a finite measure space (Theorem~\ref{maincom}).

In section 4 we    consider $2$-local derivations  on matrix
algebras over commutative regular algebras.
 We prove that every  $2$-local derivation on
the matrix algebra $M_n(\mathcal{A})$ $(n\geq 2$) over a
commutative regular algebra $\mathcal{A}$ is a derivation
(Theorem~\ref{Main}).

The main results of the Sections 3 and 4 are applied to study
$2$-local derivations on algebras of measurable and locally
measurable operators affiliated with abelian von Neumann algebras
and with type I von Neumann algebras without abelian direct
summands respectively.

\section{Commutative regular algebras}

Let $\mathcal{A}$ be a commutative algebra with the unit
$\mathbf{1}$ over the field $\mathbf{C}$ of complex numbers. We
denote by $\nabla$ the set $\{e\in \mathcal{A}: e^2=e\}$ of all
idempotents in $\mathcal{A}.$ For $e,f\in \nabla$ we set $e\leq f$
if $ef=e.$ With respect to this partial order, to the lattice
operations $e\vee f=e+f-ef, \ e\wedge f=ef$ and the complement
$e^{\bot}=\mathbf{1}-e,$ the set $\nabla$ forms a Boolean algebra.
A non zero element $q$ from the Boolean algebra $\nabla$ is called
an \textit{atom} if $0\neq e\leq q, \ e\in \nabla,$ imply that
$e=q.$ If given any nonzero $e\in \nabla$ there exists an atom $q$
such that $q\leq e,$ then the Boolean algebra $\nabla$ is said to
be \textit{atomic}.

An algebra $\mathcal{A}$ is called \textit{regular} (in the sense
of von Neumann) if for any $a\in \mathcal{A}$ there exists $b \in
\mathcal{A}$ such that $a = aba.$

Let  $\mathcal{A}$ be  a unital commutative regular algebra over
$\mathbf{C},$ and let  $\nabla$ be  the Boolean algebra of all its
idempotents. In this case given any element $a\in \mathcal{A}$
there exists an idempotent $e\in \nabla$  such that $ea=a,$ and if
$ga=a, g\in \nabla,$ then $e\leq g.$ This idempotent is called the
\textit{support} of $a$ and denoted by $s(a).$

Recall that the Boolean algebra $\nabla$  is called
\textit{complete}, if for any subset $S$  there exists the least
upper bound $\sup S\in \nabla.$ We say that a Boolean algebra
$\nabla$  is
 \textit{of countable type}, if every family of
pairwise disjoint nonzero elements from $\nabla$ is at most
countable.

Let $\mathcal{A}$ be a commutative unital regular algebra, and let
$\mu$  be a strictly positive countably additive finite measure on
the Boolean algebra $\nabla$  of all idempotents from
$\mathcal{A},$ $\rho(a,b)=\mu(s(a-b)),\ a,b\in \mathcal{A}.$ If
$(\mathcal{A}, \rho)$ is a complete metric space, then $\nabla$ is
a complete Boolean algebra of the countable type  (see
\cite[Proposition 2.7]{Ber}).

\begin{example}\label{exam}
The most important example of a complete commutative regular
algebra $(\mathcal{A},\rho)$ is the algebra
$\mathcal{A}=L^0(\Omega)=L^0(\Omega,\Sigma,\mu)$ of all (classes
of equivalence of) measurable complex functions on a measure space
$(\Omega,\Sigma,\mu),$ where $\mu$ is finite countably additive
measure on $\Sigma,$ and $\rho(a,b)=\mu(s(a-b))=\mu(\{\omega\in
\Omega:a(\omega)\neq b(\omega)\})$ (see for details
\cite[Lemma]{Ayu1} and \cite[Example 2.5]{Ber}).
\end{example}

\begin{remark} \label{rem}
If $(\Omega,\Sigma,\mu)$ is a general localizable measure space,
i.e. the  measure $\mu$  (not finite in general) has the finite
sum property, then the algebra $L^0(\Omega,\Sigma,\mu)$ is a
unital  regular algebra, but $\rho(a,b)=\mu(s(a-b))$
 is not a metric in general. But one can
represent $\Omega$ as a union of pair-wise disjoint measurable
sets with  finite measures and thus this algebra is a direct sum
of commutative regular complete metrizable algebras from the above
example. \end{remark}

 From now on we
shall assume that $(\mathcal{A},\rho)$ is a complete metric space
(cf. \cite{Ber}).

Following \cite{Ber} we say that an element $a \in \mathcal{A}$ is
\textit{finitely valued} (respectively, \textit{countably valued})
if $a = \sum\limits_{k =1}^n {\alpha _k e_k}$, where $\alpha _k
\in \mathbf{C}$, $e_k \in \nabla, \ e_k e_j=0, \ k\neq j,\ k,j
=1,...,n, \ n\in \mathbf{N}$ (respectively, $a = \sum\limits_{k
=1}^\omega {\alpha _k e_k}$, where $\alpha_k \in \mathbf{C}$, $e_k
\in \nabla, \ e_k e_j=0, \ k\neq j, \ k,j =1,...,\omega,$ where
$\omega$ is a natural number or $\infty$ (in the latter case the
convergence of series is understood with respect to the metric
$\rho$)). We denote by $K(\nabla)$ (respectively, by
$K_c(\nabla)$) the set of all finitely valued (respectively,
countably valued) elements in $\mathcal{A}.$ It is known that
$\nabla \subset K(\nabla) \subset K_c(\nabla),$ both $K(\nabla)$
and $K_c(\nabla)$ are regular subalgebras in $\mathcal{A},$ and
moreover the closure of $K(\nabla)$ in $(\mathcal{A},\rho)$
coincides with $K_c(\nabla),$ in particular, $K_c(\nabla)$ is a
$\rho$-complete (see \cite[Proposition 2.8]{Ber}).

Further everywhere we assume that $\mathcal{A}$ is a unital
commutative regular algebra over $\mathbf{C}$ and $\mu$ is a
strictly positive countably additive finite measure on the Boolean
algebra $\nabla$ of all idempotents in $\mathcal{A}.$ Suppose that
$\mathcal{A}$ is complete in the metric $\rho(a, b)=\mu(s(a-b)),\
a,b\in \mathcal{A}.$

First we recall some further notions from the paper \cite{Ber}.

Let $\mathcal{B}$ be a unital subalgebra in the algebra
$\mathcal{A}.$ An element $a\in \mathcal{A}$ is called:

-- \textit{algebraic with respect to} $\mathcal{B},$ if there
exists a polynomial $p\in \mathcal{B}[x]$ (i.e. a polynomial on
$x$ with the coefficients from $\mathcal{B}$), such that $p(a)=0$;

-- \textit{integral with respect to} $\mathcal{B},$ if there
exists a unitary polynomial $p\in \mathcal{B}[x]$ (i.e. the
coefficient of the largest degree of $x$ in $p(x)$ is equal to
$\mathbf{1}\in \mathcal{B}$), such that $p(a)=0;$

-- \textit{transcendental with respect to} $\mathcal{B},$ if $a$
is not algebraic with respect to $\mathcal{B};$

-- \textit{weakly transcendental with respect to} $\mathcal{A},$
if $a\neq 0$ and for any non-zero idempotent $e\leq s(a)$ the
element $ea$ is not integral with respect to $\mathcal{B}.$

Integral closure of a subalgebra $\mathcal{B}$ is the set of all
integral elements in $\mathcal{A}$ with respect to $\mathcal{B};$
it is denoted as $\mathcal{B}^{(i)}$. It is known (see e.g.
\cite{ber-mt}) that $\mathcal{B}^{(i)}$ is also a subalgebra in
$\mathcal{A}.$

\begin{lemma}\label{trans}
Let  $\mathcal{B}$ be a regular  $\rho$-closed subalgebra in
$\mathcal{A}$ such that
$\nabla\subset\mathcal{B}=\mathcal{B}^{(i)},$ where
$\mathcal{B}^{(i)}$ is the integral closure of $\mathcal{B}.$ Then
for every  $a\in \mathcal{A}$ there exists an idempotent
$e_a\in\nabla$ such that

i)  $e_a a\in \mathcal{B};$

ii)  if  $e_a\neq \mathbf{1}$ then  $(\mathbf{1}-e_a) a$ is a
weakly transcendental element with respect to $\mathcal{B}.$
\end{lemma}

\begin{proof}
Put
$$
\nabla_a=\{e\in\nabla: e a\in  \mathcal{B}\}.
$$
Take   $e_1, e_2\in \nabla_a,$ i.e. $e_1 a,  e_2 a\in
\mathcal{B}.$ Since  $\nabla\subset\mathcal{B}$ it follows that
$(e_1^\perp\wedge e_2) a\in \mathcal{B}.$ Further $\mathcal{B}$ is
a subalgebra and therefore
$$
(e_1\vee e_2)a=e_1 a+(e_1^\perp\wedge e_2)a\in \mathcal{B},
$$
i.e. $e_1\vee e_2\in \nabla_a.$

Denote
$$
e_a=\bigvee\nabla_a.
$$
Since  $\nabla$ is a Boolean algebra with a strictly positive
finite measure $\mu$ there exists a sequence of idempotents
$\{e_n\}_{n\in \mathbb{N}}$ in $\nabla_a$ such that
$\bigvee\limits_{n\in\mathbb{N}} e_n=e_a.$ Since    $e_n a\in
\mathcal{B}$ it follows that $\bigvee\limits_{n=1}^{m}e_na \in
\mathcal{B}$ for all $m\in \mathbb{N}.$ It is clear that
$$
\lim\limits_{m\rightarrow\infty}
\rho\left(\bigvee\limits_{n=1}^{m}e_n a, e_a a\right)\rightarrow
0.
$$
Since  $\mathcal{B}$ is  $\rho$-closed it follows that
  $e_a a\in \mathcal{B}.$

Now we suppose that  $e_a\neq \mathbf{1}.$ Let $0\neq e\leq
e_a^{\perp}.$ If  $e a$ is an integral element with respect to
$\mathcal{B}$ then by the equality $\mathcal{B}=\mathcal{B}^{(i)}$
we have that  $e a\in \mathcal{B},$ i.e. $e\leq e_a,$ which
contradicts with $0\neq e\leq e_a^{\perp}.$ Thus  $e_a^{\perp}a$
is
 weakly transcendental  with respect to $\mathcal{B}.$ The
proof is complete.
\end{proof}

Below we list some results  from \cite{Ber}, \cite{ber-mt} which
are necessary in the next section.

\begin{proposition} (see \cite[Proposition 2.3 (iv)]{Ber}).
\label{lfirst} If $\mathcal{B}$ is a subalgebra in $\mathcal{A}$
and $\delta:\mathcal{B}\rightarrow \mathcal{A}$
 is a derivation, then
$s(\delta(b))\leq s(b)$ for all $b\in \mathcal{B}.$
\end{proposition}

Recall (see \cite{Ber}) that for every element $a$ in the regular
algebra $\mathcal{A}$ there is a unique element $i(a)\in
\mathcal{A}$ such that $ai(a)=s(a).$  In particular, $\mathcal{A}$
is semi-prime. Indeed, if $a\mathcal{A}a=\{0\},$ then
$$
0=ai(a)a=s(a)a=a.
$$

\begin{proposition} (see \cite[Proposition 2.5]{Ber}).
\label{lsec} Let $\mathcal{B}$ be a subalgebra in $\mathcal{A}$
such that $\nabla\subset \mathcal{B}$ and let
$\delta:\mathcal{B}\rightarrow \mathcal{A}$ be a derivation. If
$\mathcal{B}(i)=\{a\cdot  i(b): a,  b\in  \mathcal{B}\},$ then
$\mathcal{B}(i)$ is the smallest regular subalgebra in
$\mathcal{A}$ containing $\mathcal{B},$ and there exists a unique
derivation $\delta_1:\mathcal{B}(i)\rightarrow \mathcal{A}$ such
that $\delta_1(b) = \delta(b)$ for all $b\in \mathcal{B}.$
\end{proposition}

\begin{proposition} (see \cite[Proposition 2.6 (vi)]{Ber}).
\label{lthir} If  $\mathcal{B}$ is  a subalgebra (respectively, a
regular subalgebra) in $\mathcal{A},$ such that $\nabla\subset
\mathcal{B}$ and if  $\delta:\mathcal{B}\rightarrow \mathcal{A}$
is a derivation, then the closure $\overline{\mathcal{B}}$ of the
algebra $\mathcal{B}$ in $(\mathcal{A}, \rho)$ is  a subalgebra
(respectively, a regular subalgebra) in $\mathcal{A},$ and there
exists a unique derivation
$\delta_1:\overline{\mathcal{B}}\rightarrow \mathcal{A}$ such that
$\delta_1(b) = \delta(b)$ for all $b\in \mathcal{B}.$
\end{proposition}

\begin{proposition} (see \cite[Proposition 2]{ber-mt}).
\label{lfour} Suppose that $\mathcal{B}$ is  a regular
$\rho$-closed subalgebra in $\mathcal{A},$  $\nabla\subset
\mathcal{B}$ and $\delta:\mathcal{B}\rightarrow \mathcal{A}$ is a
derivation. Let $\mathcal{B}^{(i)}$ be the integral closure of
$\mathcal{B}$ in $\mathcal{A}.$ Then $\mathcal{B}^{(i)}$ is a
regular subalgebra in $\mathcal{A}$  and there exists a unique
derivation $\delta_1:\mathcal{B}^{(i)}\rightarrow \mathcal{A}$
such that $\delta_1(b) = \delta(b)$ for all $b\in \mathcal{B}.$
\end{proposition}

We note also that for any element $a\in \mathcal{A},$ the set
$\mathcal{B}(a)=\{p(a): p\in \mathcal{B}[x]\}$ (of all polynomials
on $a$ with the coefficients from $\mathcal{B}$) is a subalgebra
in $\mathcal{A},$ which is generated by the subalgebra
$\mathcal{B}$ and the element $a.$

\begin{proposition} (see \cite[Proposition 3.6]{Ber}).
\label{lin} Let   $\mathcal{B}\subseteq \mathcal{A}$ be a regular
$\rho$-closed  subalgebra such that $\nabla\subset \mathcal{B}$
and let $\delta:\mathcal{B}\rightarrow \mathcal{A}$ be  a
derivation.
 If $a$ is an integral element
with respect to $\mathcal{B},$  then there exists a unique
derivation $\delta_1:\mathcal{B}(a)\rightarrow \mathcal{A}$ such
that $\delta_1(b)=\delta(b)$  for all $b\in \mathcal{B}.$
\end{proposition}

\begin{proposition} (see \cite[Proposition 3.7]{Ber}).
\label{lfiv} Let   $\mathcal{B}$ be a regular   subalgebra in
$\mathcal{A}$ such that $\nabla\subset \mathcal{B}$ and let
$\delta:\mathcal{B}\rightarrow \mathcal{A}$ be  a derivation. If
$a\in \mathcal{A}$ is a weakly transcendental
 element
with respect to $\mathcal{B},$  then for every $c\in \mathcal{A},$
such that $s(c)\leq s(a),$  there exists
 a unique derivation $\delta_1:\mathcal{B}(a)\rightarrow \mathcal{A},$
such that $\delta_1(b)=\delta(b)$  for all $b\in \mathcal{B}$ and
$\delta_1(a)=c.$
\end{proposition}

The following theorem  provides a sufficient condition
 for a derivation initially defined on a subalgebra of a commutative
regular algebra to have an extension to the whole algebra (see
\cite{Ber}).

\begin{theorem} (see \cite[Theorem 3.1]{Ber}).
\label{mfirst} Let $\mathcal{B}$  be a subalgebra of
$\mathcal{A}.$ Then for any derivation
$\delta:\mathcal{B}\rightarrow \mathcal{A}$ for which
$s(\delta(b))\leq s(b)$ for all $a, b\in \mathcal{B}$,  there
exists a derivation $\delta_0: \mathcal{A}\rightarrow
\mathcal{A},$
 such that $\delta_0(b) = \delta(b)$
for all $b\in \mathcal{B}.$
\end{theorem}

The main result of \cite{Ber} (Theorem 3.2)  asserts that the
algebra $\mathcal{A}$ admits a non-zero derivation if and only if
$K_c(\nabla)\neq \mathcal{A}$

Now recall the definition of algebraically independent subset over
commutative regular algebras (see for details~\cite{ber-mt}).

Let  $F[x_1,\ldots, x_n]$ be the  algebra of all polynomials of
$n$ variables over a field  $F.$ A monomial $q(x_1,\ldots, x_n)$
is said to be included to a polynomial
 $p(x_1, \ldots, x_n)\in F[x_1,\ldots, x_n],$
if   the natural  representation of  $p(x_1, \ldots, x_n)$ as a
sum of monomials with non zero coefficients contains the monomial
$q.$ Natural representation means such representation
 that any two different monomials have different degrees  of corresponding variables.
For example, for the polynomial $p(x_1,
x_2)=4x_1^3x_2^4+5x_1^2x_2^3-3x_1^2x_2^3+2$ of two variables the
natural representation is $p(x_1, x_2)=4x_1^3x_2^4+2x_1^2x_2^3+2$.

 Let  $\mathcal{A}$
be a commutative regular algebra over the field $F.$ A subset
$\mathcal{M}$ is called algebraically independent if for any $a_1,
\ldots, a_n\in \mathcal{M},$ $e\in \nabla,$ $p\in F[x_1,\ldots,
x_n]$ the equality  $ep(a_1,\ldots, a_n)=0$ implies that
$eq(a_1,\ldots, a_n)=0,$ where  $q\in F[x_1,\ldots, x_n]$ is an
arbitrary monomial included to the polynomial $p.$

If  $\mathcal{A}$ is a field then this definition coincides with
the well-known definition algebraically independence of subsets
over the field.

Note that in the case of  $f$-algebras the notion of algebraic
independence of subsets coincides with the algebraic independence
of subsets introduced by A.G. Kusraev in \cite{Kus1}.

We need the following result from   \cite[Proposition 4]{ber-mt}.

\begin{proposition}\label{trakri}
For a subset    $\{a_i: i\in I\}\subset \mathcal{A}$ the following
assertions are equivalent:
\begin{enumerate}
\item $\{a_i: i\in I\}$ is an algebraically independent subset in
$\mathcal{A};$

\item for every $i\in I$ the element $a_i$ is weakly transcendental with
respect to the algebra $\mathcal{A}_i,$ generated by $\nabla$ and
$\{a_j: j\in I, j\neq i\}.$
\end{enumerate}
\end{proposition}

\begin{lemma}\label{indep} Let $a, b\in \mathcal{A}$ and let
$s(a)=s(b)=\mathbf{1}.$  If  the subset
 $\{a, b\}$      is algebraically independent then the subset
 $\{a, a+b\}$ is also  algebraically independent.
\end{lemma}

\begin{proof}
Suppose that  $\{a, a+b\}$ is not algebraically independent. Then
by Proposition~\ref{trakri}  it follows that  $a+b$ is not weakly
transcendental element with respect to the  subalgebra $F(a,
\nabla),$ generated by the element $a$ and $\nabla.$ Therefore
there exists a non zero idempotent $e\leq s(a+b)$ such that
$e(a+b)$ is an integral element with respect to $F(a, \nabla),$
i.e. there are elements
 $c_1, c_2,\ldots, c_n\in F(a, \nabla)$
such that
$$
(e(a+b))^n+c_1(e(a+b))^{n-1}+\ldots+c_{n-1}(e(a+b))+c_n=0.
$$
By decomposing  $(e(a+b))^k,$ $k=\overline{1, n},$ the last
equality can be represented in the form
$$
(eb)^n+d_1(eb)^{n-1}+\ldots+d_{n-1}(eb)+d_n=0,
$$
where
 $d_1, d_2,\ldots, d_n\in F(a, \nabla).$
This means that  $eb$ is an integral element with respect to $F(a,
\nabla).$ Since  $s(a)=\mathbf{1}$ it follows that  $b$ is not
weakly transcendental with respect to $F(a, \nabla).$

On other hand, since $\{a, b\}$ is an algebraically independent
subset, by
 \cite[Proposition~4]{ber-mt} we have that
 $b$ is weakly transcendental with respect to $F(a, \nabla).$
From this contradiction we have that  $\{a, a+b\}$ is an
algebraically independent subset. The proof is complete.
\end{proof}

Denote by $Der(\mathcal{A})$ the set of all derivations from
$\mathcal{A}$ into $\mathcal{A}$. Let $\mathcal{M}$ be a maximal
algebraically independent subset in $\mathcal{A}$ and denote by
$K(\mathcal{M}, \mathcal{A})$ the set of all mapping
$f:\mathcal{M}\rightarrow \mathcal{A}$ such that $s(f(a))\leq
s(a)$ for every $a\in \mathcal{M}.$ The sets $Der(\mathcal{A})$
and $K(\mathcal{M}, \mathcal{A})$  equipped with natural algebraic
operations form linear spaces over $\mathbb{C}.$

 \begin{theorem} (see \cite[Theorem 1]{ber-mt}) \label{mtm}
 The map $\delta\rightarrow \delta|_{\mathcal{M}}$
 gives a linear isomorphism between  $Der(\mathcal{A})$ and $K(\mathcal{M}, \mathcal{A}).$
 \end{theorem}

\section{$2$-Local derivations on commutative regular algebras}

In this section  $\mathcal{A}$ is a unital commutative regular
algebra over $\mathbf{C},$  $\nabla$ is the Boolean algebra of all
its idempotents and  $\mu$ is a strictly positive countably
additive finite measure on  $\nabla$.  Consider the metric
$\rho(a,b)=\mu(s(a-b)),\ a,b\in \mathcal{A},$ on the algebra
$\mathcal{A}$  and from now on we shall assume that
$(\mathcal{A},\rho)$ is a complete metric space (cf. \cite{Ber}).

\begin{definition}\label{re}
For  $x\in \mathcal{A}$ denote:

--- $\mathcal{A}_0=F(x, \nabla)$  is the subalgebra in  $\mathcal{A},$
generated by  $x$ and $\nabla;$

--- $\mathcal{A}_1$ is the smallest regular subalgebra in
$\mathcal{A},$ contained $\mathcal{A}_0;$

--- $\mathcal{A}_2$ is the closure of $\mathcal{A}_1$ by the metric
$\rho;$

--- $\mathcal{A}_3$ is the integral closure of $\mathcal{A}_2;$

--- $\mathcal{A}_x$ is the closure of  $\mathcal{A}_3$ by the metric
$\rho.$
\end{definition}

Note that Proposition~\ref{lsec} provides the existence of the
subalgebra $\mathcal{A}_1.$

\begin{lemma}\label{der} If there is  an element
$a\in\mathcal{A} $ weakly transcendental with respect  to
$K_c(\nabla)$  such that $\mathcal{A}=\mathcal{A}_a,$ where
$\mathcal{A}_a$ is the
 subalgebra constructed with respect to $a$ by
 definition~\ref{re}, then any  $2$-local derivation on $\mathcal{A}$
is a derivation.
\end{lemma}

\begin{proof}
Let $a$ be a weakly transcendental element with respect  to
$K_c(\nabla)$ such that  $\mathcal{A}=\mathcal{A}_a$  and let
$\Delta$ be a $2$-local derivation on $\mathcal{A}.$ Since any
derivation on regular commutative algebra $\mathcal{A}$ does not
expand the support of elements (see Proposition~\ref{lfirst}), we
have that $s(\Delta(x))\leq s(x), x\in \mathcal{A}$.

Let us show that there exists a unique derivation $D$ on
$\mathcal{A}$ such that $D(a)=\Delta(a).$ First consider the
trivial (identically zero) derivation  $D$ on $K_c(\nabla).$ Since
$a$ is weakly transcendental with respect to $K_c(\nabla)$ and
$s(\Delta(a))\leq s(a),$  Proposition~\ref{lfiv} implies that $D$
has a unique extension (which is also denoted by $D$) onto
$\mathcal{A}_0$ such that $D(a)=\Delta(a);$ Further following
Proposition~\ref{lsec} we can extend $D$ in a unique way onto
$\mathcal{A}_1;$ and then by  Proposition~\ref{lthir}  it can be
uniquely extended onto $\mathcal{A}_2.$ Further, in view of
Proposition~\ref{lfour}  $D$ has a unique extension onto
$\mathcal{A}_3;$ and finally, applying Proposition~\ref{lthir}
once more we extend it (uniquely) onto
$\mathcal{A}_a=\mathcal{A}.$

 Now let  $x$ be
an arbitrary element from $\mathcal{A}.$ Since $\Delta$ is a
$2$-local derivation there is derivation $\delta$ (depending on
$x$ and  $a$) such that
$$
\Delta(x)=\delta(x),\, \Delta(a)=\delta(a).
$$
Thus  $\delta(a)=\Delta(a)=D(a).$ Since  $D$ is a unique
derivation on  $\mathcal{A}$ with  $\Delta(a)=D(a),$ it follows
that $\delta\equiv D.$ In particular,
$$
\Delta(x)=\delta(x)=D(x),
$$
i.e. $\Delta\equiv D.$ This means that $\Delta$ is a derivation.
The proof is complete.
\end{proof}

\begin{lemma}\label{locder} If there exist two algebraically independent
elements   $a, b\in \mathcal{A}$ with   $s(a)=s(b),$ then the
algebra $\mathcal{A}$ admits a   $2$-local derivation which is not
a derivation.
\end{lemma}

\begin{proof}  Suppose that there exist algebraically
independent elements   $a, b$ such that   $s(a)=s(b)=\mathbf{1}.$
Denote by $\mathcal{A}_a$ and  $\mathcal{A}_b$ of the subalgebras
in $\mathcal{A},$ constructed with  respect to elements $x=a$ and
$x= b,$ respectively, by definition~\ref{re}.

Since $a, b$ are algebraically independent elements,
Proposition~\ref{trakri} implies that the element $a$  is a weakly
transcendental element with respect to $K_c(\nabla).$ Similarly to
the proof of Lemma~\ref{der} we can find  a derivation $D$ on
$\mathcal{A}_a$ such that $\Delta(a)=\mathbf{1}.$ Algebraic
independence of the subset
 $\{a, b\}$  and  Proposition~\ref{trakri} imply that  $b$ is a
 weakly
 transcendental element with respect to
 the subalgebra   $\mathcal{A}_a.$
Therefore, using Proposition~\ref{lfour} once more we can extend
the derivation  $D$ with value $D(b)=\mathbf{1}$ to the subalgebra
generated by $\mathcal{A}_a$ and  $b.$ Now using
Theorem~\ref{mfirst} we can extend the  derivation  $D$ onto the
whole $\mathcal{A}.$ Hence there is a derivation  $D$ on
$\mathcal{A}$ such that
$$
D(a)=D(b)=\mathbf{1}.
$$

Now on the algebra  $\mathcal{A}$ define the operator $\Delta$ as
follows:
$$
\Delta(x)=(e_a(x)\vee e_b(x))D(x),\, x\in \mathcal{A},
$$
where  $e_a(x)$ (respectively  $e_b(x)$)  is the largest
idempotent such that
 $e_a(x) x\in \mathcal{A}_a$
(respectively $e_b(x)\in  \mathcal{A}_b$) (see Lemma~\ref{trans}).

Let us show that  $\Delta$ is a $2$-local derivation on
$\mathcal{A}$ which is not a derivation.

First we check that $\Delta$ is a  $2$-local derivation.

Take  $x, y\in \mathcal{A}.$ Consider the following three cases.

Case  1. $e_a(x)\vee e_b(x)=\mathbf{1}, e_a(y)\vee
e_b(y)=\mathbf{1}.$ Then
$$
\Delta(x)=D(x),\, \Delta(y)=D(y).
$$

Case  2. $e_a(x)\vee e_b(x)=0, e_a(y)\vee e_b(y)=0.$ Then
$$
\Delta(x)=0,\, \Delta(y)=0.
$$
Therefore for trivial derivation  $D_0$ we have that
$$
\Delta(x)=D_0(x),\, \Delta(y)=D_0(y).
$$

Case  3. $e_a(x)\vee e_b(x)=\mathbf{1}, e_a(y)\vee e_b(y)=0.$
Without loss of generality we can assume that
   $e_a(x)\neq 0.$ Since  $e_a(y)=0,$
 Lemma  \ref{trans}
implies that the element  $e_a(x) y$ is weakly transcendental with
respect to $e_a(x)\mathcal{A}_a.$ Therefore by
Proposition~\ref{lfiv}  there exists a derivation $D_1$ on the
subalgebra generated by  $e_a(x)\mathcal{A}_a$ and $e_a(x)y$ such
that
$$
D_1|_{e_a(x)\mathcal{A}_a}= D|_{e_a(x)\mathcal{A}_a},\,
D_1(e_a(x)y)=0.
$$
Now by Theorem~\ref{mfirst}
 we extend this derivation
onto the  whole $\mathcal{A},$ and denote the extension also by
$D_1.$

Similarly there exists a derivation $D_2$ on  $\mathcal{A}$ such
that
$$
D_2|_{e_b(x)\mathcal{A}_b}= D|_{e_b(x)\mathcal{A}_b},\,
D_2(e_b(x)y)=0.
$$
Put $D_3=e_a(x)D_1+e_a(x)^{\perp}D_2.$ Then
$$
\Delta(x)=D(x)=D(e_a(x)x+e_a(x)^{\perp}x)=
$$
$$
=e_a(x)D(x)+e_a(x)^{\perp}D(x)=
e_a(x)D_1(x)+e_a(x)^{\perp}D_2(x)=D_3(x)
$$
and
$$
\Delta(y)=0=e_a(x)D_1(y)+e_b(x)^{\perp}D_2(y)= D_3(y).
$$
Thus
$$
\Delta(x)=D_3(x),\, \Delta(y)=D_3(y).
$$

Now let  $x$ and  $y$ be arbitrary elements of $\mathcal{A}.$ Put
$$
e_1=e_a(x)\vee e_b(x),\, e_2=e_a(y)\vee e_b(y)
$$
and
$$
p_1=e_1\wedge e_2,\, p_2=e_1\wedge e_2^{\perp},\,
p_3=e_1^{\perp}\wedge e_1,\, p_4=(e_1\vee e_2)^{\perp}.
$$
Then $p_1+p_2+p_3+p_4=\mathbf{1}.$ Consider the restriction
$\Delta_i$  of  the $2$-local derivation $\Delta$ on
$p_i\mathcal{A},$ $i=\overline{1, 4}.$ The idempotents $e_{p_1
a}(p_1 x)\vee e_{p_1 b}(p_1 x)$ and $e_{p_1 a}(p_1 y)\vee e_{p_1
b}(p_1 y),$ corresponding  to the elements  $p_1x$ and  $p_1y,$ by
Lemma  \ref{trans}, with respect to the subalgebras
$p_1\mathcal{A}_{p_1 a},\, p_1\mathcal{A}_{p_1 b}$ are equal to
$p_1$ ($p_1$ is the unit in $p_1\mathcal{A}$).
 Therefore by the case 1 there exists a derivation
$D_1$ on  $p_1\mathcal{A}$ such that
$$
p_1\Delta(p_1x)=p_1D_1(p_1x),\, p_1\Delta(p_1y)=p_1D_1(p_1y).
$$

Similarly we consider the $2$-local derivations $p_2\Delta$ and
$p_3\Delta,$ which correspond to the case 3,
 and  the 2-local
derivation $p_4\Delta$ which corresponds to the case 2. Take the
corresponding derivations $D_i$ on  $p_i\mathcal{A},$
$i=\overline{1, 4}$ with
$$
p_i\Delta(p_ix)=p_iD_i(p_ix),\, p_i\Delta(p_iy)=p_iD_i(p_iy).
$$
Put  $D_5=D_1+D_2+D_3+D_4.$ Then
$$
\Delta(x)=D_5(x),\, \Delta(y)=D_5(y).
$$

Now we show that  $\Delta$ is not a derivation. It is sufficient
show that $\Delta$ is not an additive. Indeed, by Lemma
\ref{indep} the subset  $\{a, a+b\}$ is an algebraically
independent subset. Hence $a+b$ is a weakly transcendental element
with respect to the subalgebra generated by $a$ and $\nabla.$ Thus
$e_a(a+b)=0.$ In the same way we can show that $e_b(a+b)=0.$
Therefore
$$
\Delta(a+b)=(e_a(a+b)\vee e_b(a+b))D(x)=0.
$$

On other hand,
$$
\Delta(a)=\Delta(b)=\mathbf{1}.
$$
Thus
$$
\Delta(a)+\Delta(b)\neq \Delta(a+b).
$$
The proof is complete.
\end{proof}

The following theorem  is the main result of this section.

\begin{theorem}\label{maincom}
 Let $\mathcal{A}$ be a unital commutative
regular algebra over $\mathbf{C}$ and let $\mu$ be a strictly
positive countably additive finite measure on the Boolean algebra
$\nabla$ of all idempotents in $\mathcal{A}.$ Suppose that
$\mathcal{A}$ is complete in the metric $\rho(a,b)=\mu(s(a-b)),\
a,b\in \mathcal{A}.$ Then the following conditions are equivalent:

i)  any  $2$-local derivation on $\mathcal{A}$ is a derivation;

ii) either  $\mathcal{A}=K_c(\nabla)$ or there exists an element
$a$ weakly transcendental with respect to $K_c(\nabla)$ such that
$\mathcal{A}=\mathcal{A}_a,$ where $\mathcal{A}_a$ is the
 subalgebra constructed with respect to $a$ by
 definition~\ref{re}.
\end{theorem}

\begin{proof}
$ii)\Rightarrow i).$ First let us consider the case when
$\mathcal{A}=K_c(\nabla).$ Since each derivation on $K_c(\nabla)$
is identically zero, it follows that each $2$-local derivation on
$\mathcal{A}=K_c(\nabla)$ is also trivial. Therefore every
$2$-local derivation on  $\mathcal{A}$ is a derivation.

Now suppose that there exists an element $a\in\mathcal{A} $ weakly
 transcendental with respect to  $K_c(\nabla)$ such that
$\mathcal{A}=\mathcal{A}_a.$  Then by Lemma~\ref{der}  each
$2$-local derivation on  $\mathcal{A}$ is a derivation.

$i)\Rightarrow ii).$ Suppose that   $ii)$ is not true. Then
$\mathcal{A}\neq K_c(\nabla).$ For arbitrary $x\notin K_c(\nabla)$
put
$$
e(x)=\mathbf{1}-\bigvee\{e\in \nabla: ex\in K_c(\nabla)\}.
$$
Then
$$
 ex\notin K_c(\nabla),\, \forall\, 0\neq
e\leq e(x).
$$
Denote
$$
e_t=\bigvee\{e(x)\in \nabla: x\notin K_c(\nabla)\}.
$$
Let us  show that
$$
 e_t^{\perp}\mathcal{A}=e_t^{\perp}K_c(\nabla).
$$
Let $x\in \mathcal{A}$ be an arbitrary element. Taking into
account that $K_c(\nabla)$ is   $\rho$-complete (see Section 2),
by the definition of the idempotent $e(x)$    we have that
$e(x)^\perp x\in K_c(\nabla).$ Since $e_t^\perp \leq e(x)^\perp$
we have $e_t^\perp x\in K_c(\nabla).$
 This means that $e_t^{\perp}\mathcal{A}=e_t^{\perp}K_c(\nabla).$

Since $\nabla$ -- is a Boolean algebra with a finite measure,
 there exists a sequence $\{e(x_n)\}_{n\geq 1}$ such that
$$
e_t=\bigvee\limits_{n\geq 1} e(x_n).
$$
Put
$$
e_1=e(x_1),\, e_n=e(x_n)\wedge (e_1\vee\cdots\vee
e_{n-1})^{\perp},\, n\geq 2
$$
 and consider the element
 $$
 x_t=\sum\limits_{n\geq 1}e_nx_n.
 $$
Then we have that
$$
ex_t\notin K_c(\nabla),\, \forall\, 0\neq e\leq e_t.
$$
Indeed, let $e\leq e_t$ be an arbitrary non zero idempotent.  Take
a number $n\in \mathbb{N}$ such that $e e_n\neq 0.$ Since $e_n\leq
e(x_n),$ it follows that  $ee_n x_n\notin K_c(\nabla).$ Further,
by the equality $ee_nx_t=ee_nx_n$  we get $ee_n x_t\notin
K_c(\nabla)$ and hence $ex_t\notin K_c(\nabla).$

 Since we assumed that $ii)$ is
false, this implies that
$$
\mathcal{A}\neq \mathcal{A}_{x_{t}}.
$$
Take $y\in \mathcal{A}\setminus \mathcal{A}_{x_{t}}.$ By
Lemma~\ref{trans} there exists the largest idempotent
 $e_y$ such that $e_y y\in \mathcal{A}_{x_t}.$ Since $y\in
\mathcal{A}\setminus \mathcal{A}_{x_{t}},$ it follows that
$e_0=\mathbf{1}-e_y\neq 0.$ Again Lemma~\ref{trans} implies that
$e_0y$ is a
 weakly
 transcendental element with respect to
   $\mathcal{A}_{x_t}.$
 Denote
$$
a=e_0x_t,\, b=e_0 y.
$$
Then $s(a)=s(b).$ By construction $b$ is weakly transcendental
with respect to  $\mathcal{A}_{x_{t}},$ and hence
Proposition~\ref{trakri}  implies that the set $\{a, b\}$ is
algebraically independent.  Therefore Lemma~\ref{locder} implies
that the algebra  $\mathcal{A}$ admits    $2$-local derivation
which is not a derivation. The proof is complete.
\end{proof}

Now we can consider the problem of existence of $2$-local
derivations which are not derivations on algebras of measurable
operators affiliated with  abelian  von Neumann algebras.

It is well known that if $M$ is an abelian von Neumann algebra
with a faithful normal semifinite trace $\tau$, then $M$ is
$\ast$-isomorphic to the algebra $L^{\infty}(\Omega)=L^{\infty}
(\Omega,\Sigma,\mu)$ of all essentially bounded measurable complex
valued function on an appropriate localizable measure space
$(\Omega,\Sigma,\mu)$ and $\tau(f)=\int
\limits_{\Omega}f(t)d\mu(t)$ for  $f\in L^{\infty}
(\Omega,\Sigma,\mu).$ In this case the algebra $S(M)$ of all
measurable operators affiliated with $M$ may be identified with
the algebra $L^0(\Omega)=L^0(\Omega,\Sigma,\mu)$ of all measurable
complex valued functions on $(\Omega,\Sigma,\mu).$  In general the
algebra $S(M)$ is not metrizable.  But considering $\Omega$ as a
union of pairwise disjoint measurable sets with finite measures we
obtain that $S(M)$ is  a direct sum of commutative regular
algebras metrizable in the above sense (see Remark \ref{rem}).
Therefore using Theorem \ref{maincom}  we obtain the following
solution of the problem concerning the existence of $2$-local
derivations which are not derivations on algebras of measurable
operator in the abelian case.

\begin{theorem}\label{vonab}
Let $M$ be an abelian von Neumann algebra. The following
conditions are equivalent:

(i) the lattice $P(M)$ of projections in $M$ is not atomic;

(ii) the algebra $S(M)$ admits a $2$-local derivation which is not
a derivation.
\end{theorem}

\begin{proof}
$i)\Rightarrow ii).$ Suppose that $P(M)$ is not atomic. Then the
algebra
 $S(M)$ contains a  $\ast$-subalgebra
 $\mathcal{B}$ which is $\ast$-isomorphic with the
 $\ast$-algebra  $L^0(0, 1)$ of all measurable complex
  valued functions on  $(0, 1).$ Without loss of generality
  we may assume that  $\mathcal{B}$ contains the unit of $S(M).$ By
\cite[Lemma 2]{ber-mt} the algebra $L^0(0, 1)$ contains an
uncountable set of  algebraically independent elements, and
therefore  the algebra $\mathcal{B}$ contains algebraically
independent elements $a, b$ with $s(a)=s(b)=\mathbf{1}.$ By
Theorem~\ref{mtm}
 there exists a derivation $D$ on  $\mathcal{B}$ such that
  $D(a)=0,\, D(b)=\mathbf{1}.$
 Following  Theorem~\ref{mfirst}
 we extend $D$ onto $S(M),$  the extension is also denoted by  $D.$

Now let us show that  $\{a, b\}$ is an
 algebraically independent subset in $S(M).$
  Suppose the converse, i.e. $\{a, b\}$ is not
  algebraically independent in $S(M).$ Then by
Proposition~\ref{trakri} $b$ is not weakly transcendental with
respect to  $R(a, \nabla),$ where $R(a, \nabla)$ is the smallest
regular $\rho$-closed
 subalgebra in $S(M),$ generated by $a$ and
$\nabla.$  This means that there exists an idempotent $e$ with
$0\neq e\leq s(b)=\mathbf{1}$ such that  $eb$ is an integral
element with respect to $R(a, \nabla).$ Consider the subalgebra
$F(eb, R(a, \nabla)),$ generated by $eb$ and $R(a, \nabla).$ Let
$\delta$ denote the trivial derivation on $F(eb, R(a, \nabla)).$
By Proposition~\ref{lin}  $\delta$ is the unique derivation on
 $F(eb, R(a, \nabla))$ with  $\delta|_{R(a,
\nabla)}=0.$ Since $D(a)=0,$ it follows that $D|_{R(a,
\nabla)}=0.$ Therefore $\delta=D|_{F(eb, R(a, \nabla))},$ and in
particular, $D(eb)=0.$
 This is a contradiction with   $D(eb)=eD(b)=e\neq0.$ This
 contradiction shows that $\{a, b\}$ is an
 algebraically independent set in $S(M).$ Now
 Theorem~\ref{maincom} implies that the algebra  $S(M)$ admits a
  $2$-local derivation which is not a derivation.

$ii)\Rightarrow i).$ Suppose that   $P(M)$ is atomic. Then by
\cite[Theorem 3.4]{Ber} every derivation on $S(M)$ is identically
zero. Therefore each  $2$-local derivation on $S(M)$ is also
trivial, i.e. a derivation. The proof is complete.

\end{proof}

\section{$2$-Local   derivations
on matrix algebras}

In this section we shall investigate $2$-local derivations on
matrix algebras over commutative regular algebras.

As in the previous  section let $\mathcal{A}$ be a unital
commutative regular algebra over $\mathbf{C}$ and let $\mu$ be a
strictly positive countably additive finite measure on the Boolean
algebra $\nabla$ of all idempotents in $\mathcal{A}.$ Suppose that
$\mathcal{A}$ is complete in the metric $\rho(a, b)=\mu(s(a-b)),\
a,b\in \mathcal{A}.$

Let  $M_n(\mathcal{A})$
 be the algebra of $n\times n$ matrices over $\mathcal{A}.$
We identify the center of the algebra $M_n(\mathcal{A})$ with
$\mathcal{A}.$ If $e_{i,j},\,i,j=\overline{1, n},$ are the matrix
units in $M_n(\mathcal{A}),$ then each element $x\in
M_n(\mathcal{A})$ has the form
 $$x=\sum\limits_{i,j=1}^{n}f_{i j}e_{i j},
 \,f_{ij}\in \mathcal{A},\,i,j=\overline{1, n}.
 $$
Let  $\delta:\mathcal{A}\rightarrow \mathcal{A}$ be a derivation.
Setting
 \begin{equation}
 \label{1}
  D_{\delta}\left(\sum\limits_{i,j=1}^{n}f_{i j}e_{i j}\right)=
 \sum\limits_{i,j=1}^{n}\delta(f_{i j})e_{i j}
\end{equation}
 we obtain a well-defined linear operator
 $D_\delta$ on the algebra $M_n(\mathcal{A}).$ Moreover
 $D_\delta$ is a derivation on the algebra  $M_n(\mathcal{A})$
 and its restriction onto the center
 of the algebra  $M_n(\mathcal{A})$ coincides with the given $\delta.$

\begin{lemma}\label{str}
Let  $M_n(\mathcal{A})$
 be the algebra of $n\times n$ matrices over $\mathcal{A}.$
Every derivation  $D$ on the algebra $M_n(\mathcal{A})$ can be
uniquely represented as a sum
  $$
  D=D_{a}+D_\delta,
$$
where  $D_{a}$ is an inner derivation implemented by an element
$a\in M_n(\mathcal{A})$ while $D_\delta$ is the derivation of the
form \eqref{1} generated by a derivation $\delta$ on
$\mathcal{A}.$
\end{lemma}

In \cite[Lemma 2.2]{Alb2}  this assertion has been proved for the
case of algebras  $\mathcal{A}=L^0(\Omega),$
 but the proof is the same for general commutative regular algebras
  $\mathcal{A}.$

The proof of the following result directly  follows from the
definition of $2$-local derivations.

\begin{lemma}\label{H}
Let $\mathcal{B}$ be an algebra with the center $Z(\mathcal{B})$
and let    $\Delta: \mathcal{B}\rightarrow \mathcal{B}$ be  a
$2$-local derivation. Then  $\Delta(zx)=z\Delta(x)$ for all
central idempotent $z\in Z(\mathcal{B})$ and $x\in \mathcal{B}.$
 \end{lemma}

The following theorem  is the main result of this section.

\begin{theorem}\label{Main}
Every $2$-local derivation $\Delta: M_n(\mathcal{A})\rightarrow
M_n(\mathcal{A}),$ $n\geq2$,   is a derivation.
\end{theorem}

For the proof of the Theorem \ref{Main} we need several Lemmata.

For $x\in M_n(\mathcal{A})$ by $x_{ij}$ we denote the $(i,
j)$-entry of $x,$ i.e. $e_{ii}xe_{jj}=x_{ij}e_{ij},$ where $1\leq
i,j\leq n.$

\begin{lemma}\label{A}   For every
  $2$-local derivation $\Delta$ on  $M_n(\mathcal{A}),$ $n\geq2,$
  there exists a derivation  $D$
  such that $\Delta(e_{i j})=D(e_{i j})$ for all $i, j \in \overline{1, n}.$
\end{lemma}

\begin{proof} (cf. \cite[Theorem 3]{Kim}). We define two matrices
$d, q\in M_n(\mathcal{A})$  by
$$
d=\sum\limits_{i=1}^n \frac{1}{2^i}e_{ii},\,
q=\sum\limits_{i=1}^{n-1} e_{i, i+1}.
$$
It is easy to see that an element $x\in M_n(\mathcal{A})$
 commutes with $d$  if and only if it is diagonal,
 and if an element
$u$  commutes with $q,$ then $u$  is of the form
$$
u=\left(%
\begin{array}{cccccc}
  u_1 & u_2 & u_3 & . & .  & u_n \\
  0 & u_1 & u_2 &  . & .  & u_{n-1}  \\
  0 & 0 & u_1 &  . & .  & u_{n-2}  \\
  \vdots  & \vdots & \vdots & \vdots & \vdots & \vdots \\
   \ldots  & \ldots & \ldots & . & u_1 & u_2 \\
 0 & 0 & \ldots  & .& 0 & u_1 \\
 \end{array}%
\right).
$$
Take a derivation $D$ on $M_n(\mathcal{A})$ such that
$$
\Delta(d)=D(d),\, \Delta(q)=D(q).
$$
Replacing $\Delta$  by $\Delta-D$  if necessary, we can assume
that $\Delta(d)=\Delta(q)=0.$

Let $i, j\in \overline{1, n}.$ Take   a derivation
$D=D_h+D_\delta$ represented as in Lemma \ref{str} and such that
$$
\Delta(e_{ij})=D(e_{ij}),\,\Delta(d)=D(d).
$$
Since $\Delta(d)=0$ and $D_\delta(d)=0$ it follows that
$0=D_h(d)=hd-dh.$ Therefore $h$ has   diagonal form, i.e.
$h=\sum\limits_{i=1}^n h_ie_{ii}.$ So we have
$$
\Delta(e_{ij})=he_{ij}-e_{ij}h.
$$
In the same way starting with the element $q$ instead of $d$, we
obtain
$$
\Delta(e_{ij})=ue_{ij}-e_{ij}u,
$$
where  $u$  is of the above form, depending on $e_{ij}.$ So
$$
\Delta(e_{ij})=he_{ij}-e_{ij} h=u e_{ij}-e_{ij} u.
$$
Since
$$
he_{ij} - e_{ij} h= (h_i-h_j)e_{ij}
$$
 and
 $$
[u e_{ij}- e_{ij} u]_{ij}=0
$$
 it follows that $\Delta(e_{ij})=0.$
 The proof is complete.
\end{proof}

Further in Lemmata \ref{compo}--\ref{ma} we assume that $\Delta$
is  a  $2$-local derivation on the algebra $M_n(\mathcal{A}),$
$n\geq2,$ such that
 $\Delta(e_{ij})=0$ for all $i,  j\in \overline{1, n}.$

\begin{lemma}\label{compo}
For every $x\in M_n(\mathcal{A})$ there exist derivations
$\delta_{i j}:\mathcal{A}\rightarrow \mathcal{A},$ $i, j\in
\overline{1, n},$ such that
\begin{equation}\label{com}
\Delta(x)=\sum\limits_{i, j=1}^n\delta_{i j}(x_{i j})e_{i j}.
\end{equation}
\end{lemma}

\begin{proof}
Let $D_{x, e_{i j}}$ be  a derivation on $M_n(\mathcal{A})$ such
that
$$
\Delta(x)=D_{x, e_{i j}}(x),\, \Delta(e_{i j})=D_{x, e_{i j}}(e_{i
j}).
$$
Then
\[
e_{i j}\Delta(x)e_{i j}= e_{i j}D_{x, e_{i j}}(x)e_{i j}=
\]
\[
=D_{x, e_{i j}}(e_{i j} x e_{i j})-D_{x, e_{i j}}(e_{i j}) x e_{i
j}-  e_{i j} x D_{x, e_{i j}}(e_{i j})=
\]
\[
=D_{x, e_{i j}}(x_{j i} e_{i j})= D_{x, e_{i j}}(x_{j i})e_{i j}+
x_{j i} D_{x, e_{i j}}(e_{i j})=
\]
\[
=\delta_{j i}(x_{j i})  e_{i j},
\]
i.e.
\begin{equation*}
e_{i j}\Delta(x)e_{i j}= \delta_{j i}(x_{j i})  e_{i j}.
\end{equation*}
Multiplying the last equality from the left side by  $e_{j i}$ and
 from the right side by  $e_{j i}$ we obtain
\begin{equation*}
e_{j j}\Delta(x)e_{i i}= \delta_{j i}(x_{j i})  e_{j i}
\end{equation*}
for all $i, j\in \overline{1, n}.$ Summing these equalities for
all $i, j\in \overline{1, n},$ we get
\begin{equation*}
\Delta(x)=\sum\limits_{i, j=1}^n\delta_{i j}(x_{i j})e_{i j}.
\end{equation*}
 The proof is complete.
\end{proof}

\begin{lemma}\label{cent}
Consider the elements
$$
x=\sum\limits_{i=1}^n f_{i}e_{i i},\,  y=\sum\limits_{i=1}^n
g_{i}e_{i i},
$$
where $f_i, g_i\in \mathcal{A}$ for all $i\in \overline{1, n}.$
Then there exists a derivation $\delta$ on $\mathcal{A}$ such that
\begin{equation}\label{two}
\Delta(x)=\sum\limits_{i=1}^n \delta(f_{i})e_{i i},\,
  \Delta(y)=\sum\limits_{i=1}^n \delta(g_{i})e_{i i}.
\end{equation}
\end{lemma}

\begin{proof} By \eqref{com} for $i\neq j$ we obtain that
$$
\Delta(x)_{i j}=\Delta(y)_{i j}=0.
$$
Take a derivation $D$ on $M_n(\mathcal{A})$ such that
$$
\Delta(x)=D(x),\, \Delta(y)=D(y).
$$
By Lemma \ref{str} there exist an element $a\in M_n(\mathcal{A})$
 and a derivation $\delta:\mathcal{A}\rightarrow \mathcal{A}$ such that $D=D_a+D_\delta.$
Then
$$
\Delta(x)_{ii}=D(x)_{ii}=D_a(x)_{ii}+D_\delta(x)_{ii}=
$$
$$
=[ax-xa]_{ii}+\delta(f_{i})=\delta(f_{i})
$$
and
$$
\Delta(y)_{ii}=D(y)_{ii}=D_a(y)_{ii}+D_\delta(y)_{ii}=
$$
$$
=[ay-ya]_{ii}+\delta(g_{i})=\delta(g_{i}),
$$
because $x,\, y$ are diagonal matrices, and therefore
$$
[ax-xa]_{ii}=[ay-ya]_{ii}=0.
$$
So
\[
\Delta(x)=\sum\limits_{i=1}^n \delta(f_{i})e_{i i},\,
  \Delta(y)=\sum\limits_{i=1}^n \delta(g_{i})e_{i i}.
\]
 The proof is complete.
\end{proof}

\begin{lemma}\label{DER}
The  restriction $\Delta|_{\mathcal{A}}$  is a derivation.
\end{lemma}

\begin{proof} Put
$$
x=\sum\limits_{i=1}^n f e_{i i},\,  y=\sum\limits_{i=1}^n g e_{i
i},
$$
$$
z=\sum\limits_{i=1}^n (f+g) e_{i i},\,
w=fe_{11}+\sum\limits_{i=2}^n g e_{i i},
$$
where $f, g\in \mathcal{A}.$ Using \eqref{two} we can find
derivations $\delta_{x,w}, \delta_{y,w}, \delta_{z, w}$ on
$M_n(\mathcal{A})$ such that
$$
\Delta(x)=\sum\limits_{i=1}^n \delta_{x, w}(f)e_{i i}, \,
\Delta(y)= \sum\limits_{i=1}^n \delta_{y, w}(g)e_{i i},\,
\Delta(z)= \sum\limits_{i=1}^n \delta_{z, w}(f+g)e_{i i},
$$
$$
\Delta(w)= \delta_{x, w}(f)e_{11}+\sum\limits_{i=2}^n \delta_{x,
w}(g)e_{i i} =
$$
$$
=\delta_{y, w}(f)e_{11}+\sum\limits_{i=2}^n \delta_{y, w}(g)e_{i
i}= \delta_{z, w}(f)e_{11}+\sum\limits_{i=2}^n \delta_{z,
w}(g)e_{i i}.
$$
Then
$$
\Delta(x+y)=\Delta(z)= \sum\limits_{i=1}^n \delta_{z, w}(f+g)e_{i
i}=
$$
$$
=\sum\limits_{i=1}^n \delta_{z, w}(f)e_{i i}+ \sum\limits_{i=1}^n
\delta_{z, w}(g)e_{i i}=
$$
$$
= \sum\limits_{i=1}^n \delta_{x, w}(f)e_{i i}+ \sum\limits_{i=1}^n
\delta_{y, w}(g)e_{i i}= \Delta(x)+\Delta(y).$$ Hence
$$
\Delta(x+y)=\Delta(x)+\Delta(y).
$$
So the restriction of $\Delta$ on the $\mathcal{A}$ is an additive
and therefore,  $\Delta|_\mathcal{A}$  is a derivation.
 The proof is complete.
\end{proof}

Further in Lemmata \ref{dia}--\ref{ma}  we  assume that
$\Delta|_{\mathcal{A}}=0.$

\begin{lemma}\label{dia}
Let
\[
x=\sum\limits_{i=1}^n f_i e_{i i},
\]
where $f_i\in \mathcal{A},\, 1\leq i\leq n.$ Then $\Delta(x)=0.$
\end{lemma}

\begin{proof}
We fix a number   $k$ and take the element $y=f_k.$ By Lemma
\ref{cent} there exists a derivation $\delta$ on $\mathcal{A}$
such that
\[
\Delta(x)=D_\delta(x),\,
  \Delta(y)=D_\delta(y).
\]
Since $\Delta(y)=0$ it follows that
$$
0=\Delta(y)_{11}=\delta(f_k).
$$
i.e. $\delta(f_k)=0.$ Further
$$
\Delta(x)_{kk}=\delta(f_k)=0.
$$
Since $k$ is  an arbitrary number, it follows that
 $\Delta(x)=0.$
 The proof is complete.
 \end{proof}

In Lemmata \ref{ii}-\ref{jikk}  let $x$ be an arbitrary element
from $M_n(\mathcal{A}).$

\begin{lemma}\label{ii}
$\Delta(x)_{kk}=0$ for every  $k\in\overline{1, n}.$
\end{lemma}

\begin{proof} Let $k\in\overline{1, n}$ be a
fixed number. Put
$$
f_1=x_{kk},\, f_i=i(f_1+\mathbf{1}-s\left(f_1\right)),\, 2\leq
i\leq n.
$$
Let us verify that
 $s(f_i-f_j)=\mathbf{1},\, i\neq j.$
Note that
$$
f_i-f_1=(i-1)f_1+i(\mathbf{1}-s(f_1)),\, i>1
$$
and
$$
f_i-f_j=(i-j)(f_1+(\mathbf{1}-s(f_1))),\, i,j>1.
$$
Taking into account that $f_1$ and $\mathbf{1}-s(f_1)$ are
orthogonal we obtain
 $$
 s(f_i-f_j)=s(f_1)+s(\mathbf{1}-s(f_1))=s(f_1)+\mathbf{1}-s(f_1)=\mathbf{1}
 $$
 for all $i\neq j.$

 Now consider the element
\begin{equation}\label{diael}
y=\sum\limits_{i=1}^n f_i e_{i i},
\end{equation}
 Choose  a derivation
$D=D_a+D_\delta$ such that
$$
\Delta(x)=D(x),\,\Delta(y)=D(y).
$$
By Lemma \ref{dia} we have that   $\Delta(y)=0.$ Then
$$
0=\Delta(y)_{11}=D_a(y)_{11}+D_\delta(y)_{11}=
$$
$$
=[ay-ya]_{11}+\delta(x_{kk})=a_{11}f_1-f_1a_{11}+\delta(x_{kk})=
\delta(x_{kk}),
$$
i.e. $\delta(x_{kk})=0.$ If $i\neq j$ then
$$
\Delta(y)_{i j}=(f_i-f_j)a_{i j}.
$$
Therefore
$$
(f_i-f_j)a_{i j}=0.
$$ Since
$s(f_i-f_j)=\mathbf{1}$ it follows that $a_{i j}=0.$ So $a$ has a
diagonal form, i.e.
\begin{equation}\label{diag}
a=\sum\limits_{i=1}^n a_{i i}e_{i i}.
\end{equation}
Thus
$$
\Delta(x)_{kk}=[ax-xa]_{kk}+\delta(x_{kk})=
$$
$$
=a_{kk}x_{kk}- x_{kk}a_{kk}=0,
$$
i.e.
$$
\Delta(x)_{kk}=0.
$$
 The proof is complete.
 \end{proof}

In following two lemmata we assume that the indices $i$ and $j$
are fixed.

\begin{lemma}\label{unizero}
If $x_{ji}=\mathbf{1}\, (i\neq j)$ then $\Delta(x)_{ij}=0.$
\end{lemma}

\begin{proof}
Choose  an element $y\in M_n(\mathcal{A})$ of the form
\eqref{diael}  with $f_1=x_{ij}$ and a derivation $D=D_a+D_\delta$
such that
$$
\Delta(x)=D(x),\,\Delta(y)=D(y).
$$
Since $\Delta(y)=0$ it follows that  $a$ has the form \eqref{diag}
and $\delta(x_{ij})=0.$ Then
$$
\Delta(x)_{ji}=(a_{jj}-a_{ii})x_{ji}+\delta(x_{ji})=
(a_{jj}-a_{ii})+\delta(\mathbf{1})=a_{jj}-a_{ii},
$$
i.e.
$$
\Delta(x)_{ji}=a_{jj}-a_{ii}.
$$
On the other hand by  the equality \eqref{com}  we have that
$$
\Delta(x)_{ji}=\delta_{ji}(x_{ji})=\delta_{ji}(\mathbf{1})=0.
$$
Thus $a_{jj}=a_{ii}.$ Since $\delta(x_{ij})=0$ it follows that
$$
\Delta(x)_{ij}=(a_{ii}-a_{jj})x_{ij}+\delta(x_{ij})=0,
$$
i.e. $ \Delta(x)_{ij}=0. $
 The proof is complete.
\end{proof}

\begin{lemma}\label{jikk}
 If $i\neq j$  then $\Delta(x)_{ij}=0.$
\end{lemma}

\begin{proof}
Fix a pair $(i,j)$ and take the matrix $y\in M_n(\mathcal{A})$
such that $y_{k s}=x_{k s}$ for all $(k, s)\neq (j, i)$  and with
$y_{j i}=\mathbf{1}.$ Then
 by Lemma~\ref{unizero} it follows that $\Delta(y)_{i j}=0.$

Consider a derivation $D=D_a+D_\delta$ on $M_n(\mathcal{A})$
represented as in Lemma~\ref{str} such that
$$
\Delta(x)=D(x),\,\Delta(y)=D(y).
$$
Then
$$
\Delta(x)_{ij}=\sum\limits_{s=1}^n (a_{i s} x_{s j}-x_{i s} a_{s
j})+\delta(x_{i j})
$$
and
$$
\Delta(y)_{ij}=\sum\limits_{s=1}^n (a_{i s} y_{s j}-y_{i s} a_{s
j})+\delta(y_{i j}).
$$
By construction $y_{k s}=x_{k s}$ for all $(k, s)\neq (j, i),$ and
therefore
$$
\Delta(x)_{ij} =\Delta(y)_{ij}.
$$
But $\Delta(y)_{i j}=0,$ therefore $\Delta(x)_{i j}=0.$
 The proof is complete.
\end{proof}

Now Lemmata~\ref{ii} and~\ref{jikk} imply  the following

\begin{lemma}\label{ma}
  $\Delta\equiv 0.$
\end{lemma}

\textit{Proof of the Theorem \ref{Main}}. Let $\Delta$ be an
arbitrary  $2$-local derivation on $M_n(\mathcal{A}).$ By
Lemma~\ref{A} there exists a derivation $D$ on $M_n(\mathcal{A})$
such that
 $$
 (\Delta-D)(e_{ij})=0
 $$ for all $i,  j\in \overline{1, n}.$
Therefore by Lemma \ref{DER} $\delta=(\Delta-D)|_{\mathcal{A}}$ is
a derivation.  Consider the $2$-local derivation
$\Delta_0=\Delta-D-D_\delta.$ Then $\Delta_0(e_{ij})=0$ for all
$i,j$ and $\Delta_0|_{\mathcal{A}}=0.$ Therefore by Lemma~\ref{ma}
we get $\Delta_0=0,$ i.e. $\Delta=D+D_\delta.$
 Thus
$\Delta$ is a derivation.
 The proof of Theorem 4.3 is complete.

Let $M$ be a von Neumann algebra and denote by $S(M)$ the algebra
of all measurable operators and by $LS(M)$ - the algebra of all
locally measurable operators affiliated with $M.$ If $M$ is of
type I$_\infty$ we have proved in \cite{AKA} that every $2$-local
derivation on $LS(M)$ is a derivation. Theorem~\ref{Main} enables
us to extend this result for general type I case.

\begin{theorem}
Let  $M$ be a finite  von Neumann algebra of type I without
abelian direct summands. Then every  $2$-local derivation on the
algebra $LS(M)=S(M)$ is a derivation.
\end{theorem}

\begin{proof}
Let  $M$ be  an arbitrary finite von Neumann algebra of type I
without abelian direct summands. Then there exists a family
$\{z_n\}_{n\in F}, F\subseteq \mathbb{N} \setminus  \{1\},$ of
central projections from $M$ with $\sup z_n=\mathbf{1},$ such that
the algebra $M$ is $\ast$-isomorphic with the $C^\ast$-product of
von Neumann algebras $z_n M$ of type I$_n,$ respectively ($n\in
F$). Then
$$
z_nLS(M)=z_nS(M)=S(z_nM)\cong M_n(L^0(\Omega_n)),
$$
for appropriate measure spaces $(\Omega_n, \Sigma_n, \mu_n),$
$n\in F.$ By Lemma~\ref{H}
 we have that
$$
\Delta(z_{n}x)=z_{n}\Delta(x)
$$
for all $x\in S(M)$ and for each $n\in F.$ This implies that
$\Delta$ maps each $z_{n}S(M)$ into itself and hence induces a
$2$-local derivation $\Delta_{n}=\Delta|_{z_{n}S(M)}$ on the
algebra $S(z_nM)\cong M_n(L^0(\Omega_n))$ for each $n\in F.$ By
Theorem~\ref{Main}  it follows that the operator $\Delta_{n}$ is a
derivation for each $n\in F.$ Therefore for $x, y \in S(M)$ we
have that
$$
\Delta_n(z_n x+z_n y)=\Delta_{n}(z_n x)  +\Delta_n(z_n y)
$$
for all $n\in F.$ Again using Lemma~\ref{H}  we obtain that
$$
z_n \Delta(x+y)=\Delta_n(z_n x+z_n y)=\Delta_{n}(z_n x)
+\Delta_n(z_n y)=z_n[\Delta(x)+\Delta(y)]
$$
i.e.
$$
 z_n \Delta(x+y)=z_n[\Delta(x)+\Delta(y)]
$$
for all $n\in F.$ Since $\sup z_n=\mathbf{1}$ we get
$$
\Delta(x + y)=\Delta(x) + \Delta(y).
$$
This means that the operator  $\Delta$  is  an additive on $S(M)$
and  therefore is a derivation. The proof is complete.
\end{proof}

Now combining this theorem with the mentioned result of \cite{AKA}
we obtain

\begin{corollary}\label{finit}
If  $M$ is a type I  von Neumann algebra without abelian direct
summands then every  $2$-local derivation on the algebra $LS(M)$
is a derivation.
\end{corollary}

\begin{remark} \label{rel}
The  results  of the previous Section 3 (see Theorem~\ref{vonab})
show that in the abelian case  the properties of $2$-local
derivations on the algebra $LS(M) = S(M)$ are essentially
different.
\end{remark}

\section*{Acknowledgments}

The second  named author would like to acknowledge the hospitality
of the Chern Institute of Mathematics, Nankai University  (China).
The authors are indebted to the referee for valuable comments and
suggestions.

\end{document}